\documentclass[journal,twoside,web]{ieeecolor}
\usepackage{generic}
\usepackage{textcomp}

\newtheorem{theorem}{Theorem}[section]
\newtheorem{lemma}[theorem]{Lemma}
\newtheorem{definition}{Definition}
\newtheorem{corollary}[theorem]{Corollary}
\newtheorem{proposition}[theorem]{Proposition}

\newtheorem{remark}{Remark}
\newtheorem{assumption}{Assumption}
\newtheorem{example}{Example}

\newenvironment{pfof}[1]{\vspace{1ex}\noindent{\itshape 
	Proof of #1:}\hspace{0.5em}} {\hfill\QEDBL\vspace{1ex}}

\newcommand{\mc}{\mathcal}

\newcommand{\real}{\mathbb{R}}

\newcommand{\realpos}{\mathbb{R}_{\geq 0}}

\newcommand{\negat}{\scalebox{0.75}[.9]{\( - \)}}

\newcommand*{\QEDB}{\hfill\ensuremath{\square}}
\newcommand*{\QEDBL}{\hfill\ensuremath{\blacksquare}}
\newcommand\oprocendsymbol{\hbox{$\square$}}
\newcommand\oprocend{\relax\ifmmode\else\unskip\hfill%
\fi\oprocendsymbol}

\newcommand{\map}[3]{#1: #2 \rightarrow #3}

\newcommand{\norm}[1]{\Vert #1 \Vert}

\newcommand{\blue}[1]{{\color{blue} #1}}

\usepackage{hyperref}
\usepackage{graphicx,color}
\graphicspath{{./IMG/}{images/}{./IMG/SAGEMATH_PLOTS/}}
\usepackage{amsmath}
\usepackage{amssymb}
\usepackage{mathrsfs}
\usepackage{url}
\usepackage{cite}
\usepackage{booktabs}
\usepackage{array}
\usepackage{mathtools} 		
\usepackage{epstopdf}
\usepackage[vlined,ruled]{algorithm2e}
\usepackage{upgreek}
\usepackage{dsfont}
\usepackage{stackrel}
\usepackage{enumerate}

\DeclareMathOperator{\EX}{\mathbb{E}}
\DeclareMathOperator{\sW}{\text{subW}}
\DeclareMathOperator{\proj}{\text{proj}}

\def\BibTeX{{\rm B\kern-.05em{\sc i\kern-.025em b}\kern-.08em
    T\kern-.1667em\lower.7ex\hbox{E}\kern-.125emX}}
\title{Online Projected Gradient Descent for Stochastic Optimization with Decision-Dependent Distributions}

\author{Killian Wood, Gianluca Bianchin, and Emiliano Dall'Anese
\thanks{K. Wood is with the  Dept. of Applied Mathematics, University of Colorado Boulder. G. Bianchin and E. Dall'Anese are with the  Dept. of Electrical, Computer, and Energy Eng., University of Colorado Boulder. This work was supported in part by the National Science Foundation ERC ASPIRE and by the National Renewable Energy Laboratory through the subcontract UGA-0-41026-148.}
}

\begin{document}
\maketitle
\thispagestyle{empty} 
\begin{abstract}
This paper investigates the problem of tracking solutions of stochastic
optimization problems with time-varying costs that depend on 
random variables with decision-dependent  distributions. In this 
context, we propose the use of an online stochastic gradient descent 
method to solve the optimization, and we provide explicit bounds in 
expectation and in high probability for the distance between the 
optimizers and the points generated by the algorithm. In particular, we
show that  when the gradient error due to sampling is modeled as a
sub-Weibull random variable, then the tracking error is ultimately
bounded in expectation and in high probability.
The theoretical findings are validated via numerical simulations in the
context of charging optimization of a fleet of electric vehicles.
\end{abstract}

\begin{IEEEkeywords}
Optimization, Optimization algorithms. 
\end{IEEEkeywords}

\section{Introduction}

\IEEEPARstart{T}{his} paper considers the problem of developing and analyzing online algorithms to track the solutions of time-varying stochastic optimization problems, where the distribution of the underlying random variables is decision-dependent. Formally, we consider
problems of the form\footnote{\textit{Notation.} We let $\mathbb{N}_0 := \mathbb{N} \cup \{0\}$, where $\mathbb{N}$
denotes the set of natural numbers. For a given column vector $x \in \mathbb{R}^n$, $\| x\|$ is the Euclidean norm. Given a differentiable function $f: \mathbb{R}^n \rightarrow \mathbb{R}$, $\nabla f(x)$ denotes the gradient of $f$ at $x$ (taken to be a column vector). Given a closed convex set
$C \subseteq \mathbb{R}^n$, $\proj_{C}:\mathbb{R}^n \to \mathbb{R}^n$ denotes the 
Euclidean projection of $y$ onto $C$, namely $\proj_{C} (y) := \arg \min_{v \in C} \norm{y-v}$. For a given random variable $z \in \mathbb{R}$, $\mathbb{E}[z]$ denotes the expected value of $z$, and $\mathbb{P}(z \leq \epsilon)$ denotes the probability of $z$ taking values smaller than or equal to $\epsilon$; $\|z\|_p := \mathbb{E}[|z|^p]^{1 / p}$, for any $p \geq 1$. Finally, $e$ denotes Euler's number.
}:
\begin{equation}
\label{problemStatement}
x^{*}_{t} \in\arg\min_{x\in C_{t}} \underset{z\sim D_{t}(x)}{\EX} \left[ \ell_{t} (x,z)  \right],
\end{equation}
where $t\in\mathbb{N}_{0}$ is a time index, $x \in \real^d$ is 
the decision variable, $D_{t}$ is a map from the set $\real^d$ to the space of 
distributions, $z \in \mathcal{Z}_t$ is a random variable (with $\mathcal{Z}_t$ the union of the support of $D_{t}(x)$ for all $x \in C_t$),
$\map{\ell_{t}}{\real^d \times \mathcal{Z}_t}{\real}$ is the loss function, and $C_{t}\subseteq\real^{d}$ is a closed and convex set. 
Problems of this form arise in sequential learning and 
strategic classification~\cite{wilson2019adaptive}, and in 
applications in power and energy 
systems~\cite{mathieu2011examining,tushar2012economics} to 
model uncertainty in pricing and human behavior.  
Moreover, the framework \eqref{problemStatement} can be used
to solve control problems for dynamical systems whose 
dynamics are  unknown, where the variable $z$  is used to account for the lack of 
knowledge of the underlying system dynamics (similarly to  
problems in feedback-based 
optimization~\cite{hauswirth2021optimization,bianchin2021online}).

Since the distribution of $z$ in~\eqref{problemStatement} 
depends on the decision variable $x$, the problem of finding $x_t^*$ is computationally burdensome for general cases, and intractable when $D_t$ in unknown -- even when the loss function is convex 
in $x$~\cite{perdomo2020performative,drusvyatskiy2020stochastic}. 
For this reason, we focus on finding decisions that are optimal with 
respect to the distribution that they induce; we refer to these points 
as \emph{performatively stable}~\cite{perdomo2020performative}, while 
we refer to solutions $x_t^*$ to the original 
problem~\eqref{problemStatement} as \emph{performatively  optimal}. 
We obtain explicit error bounds between performatively optimal
and performatively stable  points by leveraging tools 
from~\cite{perdomo2020performative,drusvyatskiy2020stochastic}.
The main focus of this work is to propose and analyze online algorithms
that can determine performatively stable points, in contexts where the 
loss function and constraint set are revealed sequentially.  
Since the distributional map $D_{t}$ 
may be unknown in practice, we then extend our techniques to stochastic
methods that only require samples of $z$.

\textit{Prior Work:} Online (projected) gradient descent 
methods have been well-investigated by using tools 
form the controls community, we refer to the 
representative works~\cite{popkov2005gradient,selvaratnam2018Numerical,mokhtari2016online,simonetto2017time,madden2021bounds} as well as 
to pertinent references therein. Convergence guarantees for  online 
stochastic gradient methods where drift and noise terms satisfy 
sub-Gaussian assumptions were recently provided
in~\cite{cutler2021stochastic}. Online stochastic optimization problems with time-varying distributions are studied in, e.g., \cite{wilson2019adaptive,cao2020online,shames2020online}. On the other hand,  time-varying costs are considered in \cite{CW-VV-AN:18}, along with sampling strategies to satisfy regret guarantees. For static optimization problems, the notion of performatively stable points is introduced in \cite{perdomo2020performative}, where error bounds for risk minimization and gradient descent methods applied to stochastic problems with decision-dependent distributions are provided. Stochastic gradient methods to identify performatively stable points for decision-dependent distributions are studied in \cite{drusvyatskiy2020stochastic,mendler2020stochastic} -- the latter also providing results for an online setting in expectation. 
A stochastic gradient method 
for time-invariant distributional maps is presented in \cite{izzo2021learn}. 

\textit{Contributions:} 
We offer the following main contributions. C1) First, we propose an 
online projected gradient descent (OPGD) method to 
solve~\eqref{problemStatement}, and we show that the tracking error 
(relative to the performatively stable points) is ultimately 
bounded by terms that account for the temporal drift of the 
optimizers.
C2) Second, we propose an online stochastic projected gradient descent 
(OSPGD) and we provide error bounds in \emph{expectation} and in
\emph{high probability}. 
Our bounds in high probability are derived by modeling the 
gradient error as a  sub-Weibull random 
variable~\cite{vladimirova2020sub}: this allows us to capture a 
variety of sub-cases, including scenarios where the error follows 
sub-Gaussian and sub-Exponential distributions~\cite{Vershynin}, or 
any distribution with finite support.

Relative to \cite{wilson2019adaptive,cao2020online,shames2020online,CW-VV-AN:18} our distributions are decision dependent; relative to \cite{perdomo2020performative,drusvyatskiy2020stochastic,mendler2020stochastic,izzo2021learn}, our cost and distributional maps are time varying. 
Moreover, our results do not rely on bias or variance assumptions regarding the gradient estimator. In the absence of distributional shift and without a sub-Weibull error, our upper bounds reduce to the results of \cite{selvaratnam2018Numerical}. Relative to \cite{izzo2021learn}, we seek performatively stable points rather than the performative optima. In doing so, we incur the error characterized in \cite{perdomo2020performative}; however, we do not restrict to distributional maps that are continuous distributions or require finite difference approximations.
With respect to the available literature on stochastic 
optimization, we provide for the first time explicit 
bounds in expectation and in high probability to solve
stochastic optimization with decision dependent 
distributions in the presence of time-dependent 
distributional maps.

The remainder of this paper is organized as follows. 
Section~\ref{sec:preliminaries} introduces some preliminaries; 
Section~\ref{sec:OPGD} studies the OPGD, and Section~\ref{sec:OSPGD} 
studies the OSPGD. Section~\ref{sec:results} illustrates simulation 
results, and Section~\ref{sec:conclusions} concludes the~paper. 

\section{Preliminaries}
\label{sec:preliminaries}
We first introduce preliminary definitions and results. We consider random variables $z$ that take values on a metric space 
$(M,d)$, where the set $M$ is equipped with the Borel $\sigma$-algebra 
induced by metric $d$. We assume that $M$ is a complete and separable metric space (hence $M$ is a Polish space). We let 
$\mathcal{P}(M)$ denote the set of Radon probability measures on $M$ 
with finite first moment. 
Given $\nu \in \mathcal{P}(M)$, $z \sim \nu$ denotes that the
random variable $z$ is distributed according to $\nu$.
Due to Kantorovich-Rubenstein duality, the 
Wasserstein-1 distance between  $\mu,\nu\in \mathcal{P}(M)$
can be defined as~\cite{kantorovich1958space}:
\begin{equation}
\label{def:Wasserstein}
W_{1}(\mu,\nu) = \underset{g\in \text{Lip}_{1}}{\sup}\left\{\underset{z\sim\mu}{\EX} \left[g(z)\right]-\underset{z\sim\nu}{\EX}  \left[g(z)\right] \right\},
\end{equation}
where $\text{Lip}_{1}$ is the set of 1-Lipschitz functions over $M$. We note that the pair $(\mathcal{P}(M),W_{1})$ describes a metric space 
of probability measures.

\textit{Heavy-Tailed Distributions.}
In this paper, we will utilize the sub-Weibull model~\cite{vladimirova2020sub}, introduced next.

\begin{definition}{(\textbf{\textit{Sub-Weibull Random Variable}})}
\label{def:subweibull}
$z$ is a sub-Weibull random variable, denoted by 
$z\sim \sW(\theta,\nu)$, if there  
exists $\theta,\nu > 0$ such that $\norm{z}_{k}\leq \nu k^{\theta}$ for all $k\geq1$.
\end{definition}

\smallskip

The parameter $\theta$ measures the heaviness of the tail (higher values correspond to 
heavier tails) and the parameter $\nu$ measures the 
proxy-variance~\cite{vladimirova2020sub}. 
In what follows, we will also use the following equivalent 
characterization of a sub-Weibull random variable: 
$z\sim \sW(\theta,\nu)$ if and only if
$\exists~\theta, \nu' > 0, \mathbb{P}(|z| \geq \epsilon ) \leq 2\exp(-\left(\epsilon/\nu')\right)^{1/\theta}$.
As shown in \cite{wong2020lasso}, the two characterizations are 
equivalent by choosing $\nu=\left(\frac{\theta}{2e}\right)^{\theta}\nu'$.
The class of sub-Weibull random variables enjoys the following properties. 

\begin{proposition}{(\textbf{\textit{Closure of Sub-Weibull}})} \label{prop:closureSubWeibull}
Let $z \sim \sW(\theta_1,\nu_1)$ and $y \sim \sW(\theta_2,\nu_2)$ be (possibly coupled) sub-Weibull random variables and let  $c\in\mathbb{R}$. Then, the following holds:
\begin{enumerate}
    \item  $z+y\sim\sW(\max \{\theta_1,\theta_2\}, \nu_1+\nu_2)$;
    \item  $zy\sim \sW(\theta_1 + \theta_2, 
    \psi(\theta_1, \theta_2) \nu_1 \nu_2)$, $\psi(\theta_1, \theta_2) := (\theta_1 + \theta_2)^{\theta_1 + \theta_2} / (\theta_1^{\theta_1} \theta_2^{\theta_2})$;
    \item $z+c\sim \sW(\theta_1, |c| + \nu_1)$;
    \item $cz\sim \sW(\theta_1, |c|\nu_1)$. 
\end{enumerate}

\end{proposition}

\begin{proof}
Properties 1) and 4) are proved in \cite{vladimirova2020sub}; property 2) is proved in~\cite{bastianello2021stochastic}. To show 3), since $c\in\mathbb{R}$, then for any $k\geq 1$ 
$\|c\|_{k} = |c| \leq |c|k^{\theta}$.
It follows that $\|z+c\|_{k} \leq \|z\|_{k} + \|c\|_{k} \leq \nu k^{\theta} + |c|k^{\theta} \leq (\nu+|c|)k^{\theta}$.
\end{proof}

\section{Online Projected Gradient Descent}
\label{sec:OPGD}
In this section, we propose and study an OPGD method to 
solve~\eqref{problemStatement}. In Section~\ref{sec:OSPGD}, we will 
leverage the results derived in this section to analyze a the 
stochastic version OSPGD.

We begin by outlining our main  assumptions.

\begin{assumption}(\textbf{\textit{Strong Convexity}})
\label{as:strongConvexity}
For a fixed $z \in \mathcal{Z}_t$, the map 
$x \mapsto \ell_{t}(x,z)$ is $\alpha_{t}$-strongly convex, where
$\alpha_t>0$, for all $t \in \mathbb{N}_0$.
\end{assumption}

\begin{assumption}(\textbf{\textit{Joint smoothness}})
\label{as:jointSmoothness}
For all $t \in \mathbb{N}_0$,
$ x\mapsto\nabla_{x}\ell_{t}(x,z)$ is $\beta_{t}$-Lipschitz continuous for
all $z \in \mathcal{Z}_t$, and $z\mapsto\nabla_{x}\ell_{t}(x,z)$ is 
$\beta_{t}$-Lipschitz continuous for all $x \in \real^d$.
\end{assumption}

\begin{assumption}(\textbf{\textit{Distributional Sensitivity}})
\label{as:distributionalSensitivity}
For all $t \in \mathbb{N}_0$, there exists $\varepsilon_{t}>0$ such that
\begin{align}
W_{1}(D_{t}(x),D_{t}(x')) \leq \varepsilon_{t} \|x - x'\|_{2}
\end{align}
for any $x,x'\in\mathbb{R}^d$. \QEDB
\end{assumption}

\begin{assumption}(\textbf{\textit{Convex Constraint Set}})
\label{as:convexConstraints}
For all $t\in\mathbb{N}_{0}$, the set $C_{t}$ is closed and convex.
\QEDB
\end{assumption}

\subsection{Performatively stable points}
Since the objective function and the distribution in
\eqref{problemStatement} both depend on the decision variable $x$, the problem~\eqref{problemStatement} is intractable in general, even when the loss is convex.
For this reason, we follow the approach of~\cite{perdomo2020performative,drusvyatskiy2020stochastic} and seek optimization 
algorithms that can determine the performatively stable point, defined as 
follows:
\begin{equation}
\label{performativelyStable}
\bar{x}_{t} \in \arg\min_{x\in C_{t}} \underset{z\sim D_{t}(\bar{x}_{t})}{\EX} \left[ \ell_{t} (x,z)  \right] \, . 
\end{equation}
Convergence to a performatively stable  point is desirable because it 
guarantees that $\bar{x}_{t}$ is optimal for the distribution that it  
induces on $z$. The following result, adapted 
from~\cite[Prop. 3.3]{drusvyatskiy2020stochastic}, establishes 
existence and uniqueness of a performatively stable  point.

\begin{lemma}{\textbf{\textit{
(Existence of Performatively Stable Points) \textit{\cite[Prop. 3.3]{drusvyatskiy2020stochastic}}})}}
\label{lem:StableOptima}
Let Assumptions \ref{as:strongConvexity}-\ref{as:convexConstraints} 
hold, and suppose  that 
$\frac{\varepsilon_{t}\beta_{t}}{\alpha_{t}}<1$ for all 
$t \in \mathbb{N}_0$. Then, a sequence of performatively stable 
points $\{\bar{x}_{t}\}_{t \in \mathbb{N}_0}$ exists and is unique. 
\end{lemma}

In general, performatively stable points may not coincide with the optimizers of the original problem \eqref{problemStatement}. However, an explicit error 
bound can be derived, as formally stated next.

\begin{lemma}{\textbf{\textit{
(Error of Performatively Stable Points \cite{perdomo2020performative})}}}
\label{lem:distanceStableOptima}
Suppose that the function $z\mapsto\ell_{t}(x,z)$ is  $\gamma_{t}$-Lipschitz continuous
for all $x \in \real^d$ and 
$t\in \mathbb{N}_0$. Then, under the same assumptions of Lemma~\ref{lem:StableOptima}, it holds that 
\begin{equation}
\label{boundedDistance}
    \|\bar{x}_{t} - x_{t}^{*}\| \leq 2 \varepsilon_{t}\gamma_{t} \alpha_{t}^{-1}, \text{ for all } \in \mathbb{N}_0.
\end{equation}
\end{lemma}
The proof Lemma \ref{lem:distanceStableOptima} follows from
\cite[Thm 3.5, Thm 4.3]{perdomo2020performative}.
In the remainder of this paper, we assume that the assumptions of
Lemma~\ref{lem:StableOptima} are satisfied, so that the performatively 
stable point sequence is unique.
We illustrate the difference between $\bar{x}_{t}$ and 
$x_{t}^{*}$ in the following example.

\begin{example}
\label{ex:distinct}
Consider an instance of~\eqref{problemStatement} where 
$\ell(x,z) = x^2 + z,$ $C_t = \real$,  
$D_t(x) = \mathcal{N}(\mu_t x,\sigma_t^{2})$, $\mu_t, \sigma_t>0$. 
In this case, the objective can be specified in closed form as: 
$\underset{z\sim D_t(x)}{\EX} \left[x^2 + z \right] = x^2 + \mu_{t}x$, 
and thus the unique performatively optimal point is  given by
$x_t^{*} = -\mu_t/2$. 
To determine the performatively stable point, notice that 
$\nabla_x \ell(x,z)  = 2x$, and thus $\bar{x}_t$ satisfies
$\underset{z\sim D_{t}(\bar{x}_{t})}{\EX} \left[2\bar{x}_t
\right] = 0$, which 
implies $\bar{x}_t=0$. The bound in \eqref{boundedDistance} thus holds 
by noting that $\varepsilon_{t}=\mu_{t}$, $\gamma_t=1$, and 
$\alpha_t=2$.
\QEDB
 \end{example}

\subsection{Online projected gradient descent}
We now propose an OPGD that seeks to track the trajectory of 
the performatively stable optimizer 
$\{\bar{x}_{t}\}_{t \in \mathbb{N}_0}$.
To this end, in what follows we adopt the following  notation:
\begin{equation}
f_{t}(x,\nu) :=  \underset{z\sim \nu}{\EX} \left[ \ell_{t} (x,z)  \right],
\end{equation}
for any $x\in\real^{d}$, $\nu\in\mathcal{P}(M)$, and 
$t \in \mathbb{N}_0$.
Notice that when $\nu$ is a distribution induced by the decision 
variable $y$, namely $\nu=D_{t}(y)$, we will use the notation 
$f_{t}(x,D_{t}(y))$. Moreover, we denote by $\nabla f_{t}(x,\nu)$ the gradient of $f_{t}(x,\nu) $ (we also note  that, according to the  dominated convergence theorem, the expectation and 
gradient operators can be interchanged).

The OPGD amounts to the following step at each $t \in \mathbb{N}_0$:
\begin{align}
\label{eq:OPGD}
x_{t+1} &= G_{t}(x_t, D_{t}(x_t)), \,\,\,\, 
\end{align}
where $G_{t}(x_t,\nu) := \proj_{C_{t}}\left(x_t - \eta_{t}\nabla f_{t}(x_t,\nu) \right)$, with $\eta_{t} >0$ denoting a stepsize.

First, we note that a performatively stable point is a fixed
point of the algorithmic map \eqref{eq:OPGD}, namely, 
$\bar{x}_{t} = G_{t}(\bar{x}_{t},D_{t}(\bar{x}_t))$.
Next, we focus on characterizing the error between the 
updates~\eqref{eq:OPGD} and the performatively stable points 
$\{\bar x_t\}_{t\in \mathbb{N}_0}$.
To this aim, we denote the temporal drift in the performatively stable points as $ \varphi_{t} := \|\bar{x}_{t+1}-\bar{x}_{t}\|,$ and the tracking error relative to the performatively stable points as  $e_{t} := \|x_{t}-\bar{x}_{t}\|$.  Our error bound for  OPGD is presented~next.

\vspace{.1cm}

\begin{theorem}{\textbf{\textit{(Tracking Error of OPGD)}}}
\label{thrm:OPGD}
Let Assumptions \ref{as:strongConvexity}-\ref{as:convexConstraints}
hold, suppose  that $\frac{\varepsilon_{t}\beta_{t}}{\alpha_{t}}<1$ for
all $t \in \mathbb{N}_0$, and let $\{x_t\}_{t\in \mathbb{N}_0}$ 
denote a sequence generated by \eqref{eq:OPGD}. Then, for all 
$t\in \mathbb{N}_0$, the error $e_t = \|x_{t}-\bar{x}_{t}\|$ satisfies:
\begin{equation}
\label{RRM_contr}
e_{t+1} \leq \ a_{t}e_{0} + \sum _{i = 0 }^{t} b_{i} \varphi_{i},
\end{equation}
where $a_{t} := \prod_{i=1}^{t} \rho_{i}+\eta_{i}\beta_{i}\varepsilon_i$, 
\[
  b_{i} :=
  \begin{cases}
                                   1 & \text{if $i=t$}, \\
  \prod_{k=i+1}^{t} \rho_{k} + \eta_{k}\beta_{k}\varepsilon_{k} & \text{if $i\neq t$},
  \end{cases}
\]
and $\rho_{t} := \max\{ |1-\eta_{t}\alpha_{t} |, |1-\eta_{t}\beta_{t} | \}$. 
Moreover, if 
\begin{align}
\label{stepSize}
\eta_{t} \in \bigg[ \frac{1-r}{\alpha_{t}+\beta_{t}\varepsilon_{t}} , \frac{1+r}{\beta_{t}(1+\varepsilon_{t})} \bigg]
\text{ for all } t\in \mathbb{N}_0,
\end{align}
for some $r \in (0,1)$, then 
$\tilde \lambda:=\sup_{t \geq 0} \{\rho_{t}+\eta_{t}\beta_{t}\varepsilon_t\}\leq r$ 
and 
\begin{align}
\label{eq:limsup}
\limsup_{t \rightarrow +\infty} e_{t} \leq (1 - \tilde \lambda)^{-1} 
\sup_{t \geq 0} \{\varphi_t\},
\end{align}
where $\tilde{\varphi}:=\sup_{t\geq0} \{\varphi_{t}\}$.
\end{theorem}

\vspace{.1cm}

Before presenting the proof, some remarks are in order.

\begin{remark}
By application of Lemma \ref{lem:distanceStableOptima}, OPGD guarantees
that the error between the algorithmic updates and the performatively 
optimal points is bounded at all times. Precisely, the following 
estimate holds:
  $\limsup_{t \rightarrow +\infty} \|x_t - x_t^*\| \leq (1 - \tilde{\lambda})^{-1}\tilde{\varphi} + 2 \sup_{t \geq 0} \{\varepsilon_{t}\gamma_{t} \alpha_{t}^{-1}\}.$
\QEDB
\end{remark}

\vspace{.1cm}

\begin{remark}
When~\eqref{stepSize} holds, one can write the bound $e_{t+1} \leq a_{t}e_{0} + (1 - \tilde{\lambda})^{-1} \sup_{i}\{ \varphi_{i}\}$; this is  an exponential input-to-state-stability (E-ISS) result~\cite{jiang2001input}, where $\{\bar{x}_t\}$ are the equilibria of~\eqref{eq:OPGD} and $\varphi_{i}$ is treated as a disturbance. ISS implies that $e_t$ is ultimately bounded as in~\eqref{eq:limsup}. 
\QEDB
\end{remark}

Next, we present the proof of Theorem~\ref{thrm:OPGD}. 
The following lemmas are instrumental.

\begin{lemma}(\textbf{\textit{Gradient Deviations}})
\label{lem:deviation}
Under Assumption 2, for any $t\in \mathbb{N}_0$, $x\in\real^{d}$, and  
measures $\mu,\nu\in\mathcal{P}(M)$, the following bound holds:
\begin{equation}
\label{eq:deviation}
\|\nabla f_{t}(x,\mu) - \nabla f_{t}(x,\nu) \| \leq \beta_{t}W_{1}(\mu,\nu).
\end{equation}
\end{lemma}

\begin{lemma}(\textbf{\textit{Contractive Map}})
\label{lem:OPGD}
Let Assumptions \ref{as:strongConvexity}-\ref{as:jointSmoothness} 
and \ref{as:convexConstraints} hold.
For any $\nu\in\mathcal{P}(M)$, the map $x\mapsto G_{t}(x,\nu)$ is 
Lipschitz continuous, namely, for any $x,y\in\mathbb{R}^d$:
\begin{equation}
\| G_{t}(x,\nu) - G_{t} (y,\nu) \Vert \leq 
\rho_{t} \Vert x-y \|,
\end{equation}
where 
$\rho_{t} = \max\{ |1-\eta_{t}\alpha_{t} |, |1-\eta_{t}\beta_{t} | \}$. 
Moreover, if $\rho_{t}<1$ for all $t \in \mathbb{N}_0$, then 
$\bar{x}_{t}$ is the unique fixed point of \eqref{eq:OPGD}.
\end{lemma}

The proof of Lemma~\ref{lem:deviation} follows by iterating  the reasoning in \cite[Lemma 2.1]{drusvyatskiy2020stochastic} for all 
$t \in \mathbb{N}_0$; the proof of lemma~\ref{lem:OPGD} is standard and is omitted due to space limitations. 

\begin{pfof}{Theorem~\ref{thrm:OPGD}}
Note that $x_t \in C_t$ for all $t \in \mathbb{N}_{0}$ directly follows by definition of 
Euclidean projection. By using the triangle inequality, we find that
\begin{align*}
e_{t+1} 
& \leq \|x_{t+1} - \bar{x}_{t}\| + \|\bar{x}_{t} - \bar{x}_{t+1}\| \\
&   =  \| G_{t}(x_{t},D_{t}(x_t)) - G_{t}(\bar{x}_{t}, D_{t}(\bar{x}_{t}))\| + \varphi_{t} \\ 
&\leq \| G_{t}(x_{t},D_{t}(x_t)) - G_{t}(x_{t}, D_{t}(\bar{x}_{t}))\|\\ 
& \qquad + \| G_{t}(x_{t},D_{t}(\bar{x}_{t})) - G_{t}(\bar{x}_{t}, D_{t}(\bar{x}_{t}))\| + \varphi_{t},
\end{align*}
where the first identity follows from the definition of 
$G_t(\cdot,\cdot)$ and the second inequality follows by adding and 
subtracting $G_{t}(x_{t}, D_{t}(\bar{x}_{t}))$.
Applying~\eqref{eq:deviation} and Lemma \ref{lem:OPGD} yields:
\begin{align}
\label{eq:aux_Gbound}
e_{t+1}
& \leq \eta_{t} \|\nabla f_{t}(x_{t},x_{t})  - \nabla f_{t} (x_{t},\bar{x}_{t} ) \| \nonumber\\
& \qquad + \| G_{t}(x_{t},D_{t}(\bar{x}_{t})) - G_{t}(\bar{x}_{t}, D_{t}(\bar{x}_{t}))\| + \varphi_{t} \nonumber\\ 
& \leq  \eta_{t}\beta_{t} W_{1}(D_{t}(x_t),D_{t}(\bar{x}_t)) + \rho_{t} e_{t} + \varphi_t \nonumber\\ 
& \leq \eta_{t}\beta_{t}\varepsilon_{t}e_{t} + \rho_{t}e_{t} +\varphi_{t} \nonumber\\ 
& = (\rho_{t} + \eta_{t}\beta_t\varepsilon_{t})e_{t} + \varphi_{t}.
\end{align}
Thus we obtain the following by expanding the recursion:
\begin{align*}
e_{t+1} & \leq \left(\prod_{i=0}^{t} \lambda_{i} \right)e_{0} + \varphi_{t} + \sum_{i=0}^{t-1} \left ( \prod_{k=i+1}^{t} \lambda_{k} \right) \varphi_{i},
\end{align*}
where we defined
$\lambda_{t} := \rho_{t} + \eta_{t}\beta_{t}\varepsilon_{t}$. The bound \eqref{RRM_contr} then follows by definition of the sequences $\{a_{t}\}$ and $\{b_{t}\}$.

To prove \eqref{eq:limsup}, we show that $\sup_{t}\lambda_{t}<1$ for appropriate $\eta_{t}$. Fix $r\in(0,1)$. Then, by the definition of $\rho_t$, $\lambda_{t}\leq r$ 
holds if the following two conditions are satisfied simultaneously:
\begin{align}
\label{eqn:doubleIneq}
|1-\eta_{t}\alpha_{t} | + \eta_{t}\beta_{t}\varepsilon_{t} \leq r,  
\  \text{and} \ 
|1-\eta_{t}\beta_{t} | + \eta_{t}\beta_{t}\varepsilon_{t} &\leq r.
\end{align}
The first inequality holds if and only if 
$- r +  \eta_{t}\beta_{t}\varepsilon_{t}  <  1-\eta_{t}\alpha_{t} < r - \eta_{t}\beta_{t}\varepsilon_{t}$ or, equivalently,  
$1- r \leq \eta_{t}(\alpha_{t}+\beta_{t}\varepsilon_{t}) \leq 1+r$. 
The second inequality holds if and only if
$ -r + \eta_{t}\beta_{t}\varepsilon_{t} < 1- \eta_{t}\beta_{t} < r - \eta_{t}\beta_{t}\varepsilon_{t}$. 
By using $\alpha_{t}\leq \beta_{t}$, both inequalities are satisfied when 
\begin{equation*}
    \frac{1-r}{\beta_{t}(1+\varepsilon_{t})}\leq \frac{1-r}{\alpha_{t}+\beta_{t}{\varepsilon_{t}}} \leq \eta_{t} \leq  \frac{1+r}{\beta_{t}(1+\varepsilon_{t})}\leq \frac{1+r}{\alpha_{t}+\beta_{t}\varepsilon_{t}}.
\end{equation*}
Thus, to satisfy the maximum, its sufficient to enforce that $\eta_{t} \in \bigg[ \frac{1-r}{\alpha_{t}+\beta_{t}\varepsilon_{t}} , \frac{1+r}{\beta_{t}(1+\varepsilon_{t})} \bigg] $. The result~\eqref{eq:limsup} follows by utilizing the geometric series.
\end{pfof}

Finally, we observe that when the objective and constraints are time-invariant, we 
recover the result of \cite[Sec. 5]{drusvyatskiy2020stochastic} as 
formalized next.

\begin{corollary}{\bf \textit{(Tracking Error of OPGD for 
Time-Invariant Problems)}}
\label{cor:linearConvergence}
If the problem \eqref{problemStatement} is time independent and the 
assumptions in Theorem \ref{thrm:OPGD} hold, then OPGD with fixed step 
size $\eta\in(0,2/\beta(1+\varepsilon))$  converges linearly to the 
performatively stable point. 
\end{corollary}
\begin{proof}
When \eqref{problemStatement} is time independent, then 
for all $t\in\mathbb{N}_0$, $\alpha_{t}=\alpha$, $\beta_{t}=\beta$, $\varepsilon_{t}=\varepsilon$, $\rho_{t}=\rho$, $\varphi_{t}=0$. 
Accordingly, the recursion \eqref{RRM_contr} yields: 
$e_{t+1}\leq \lambda e_{t}$ with $\lambda=\rho+\eta\beta\varepsilon$. 
By replacing strict inequality and $r=1$ in \eqref{eqn:doubleIneq}, we conclude that $\eta<2/\beta(1+\varepsilon)$ implies 
$\lambda<1$. Hence $e_{t+1}/e_{t}\leq \lambda <1$.
\end{proof}

\section{Online Stochastic Gradient Descent}
\label{sec:OSPGD}
An exact expression for the distributional map 
$x_t \mapsto D_{t}(x_t)$ may not be available in general and, even if available, evaluating the gradient may be computational burdensome. We consider the case where we have access to a finite number of 
samples of $z_t$ at each time step $t$ to estimate the gradient $\nabla f_{t}(x_t,D_t(x_t))$. For example, given a mini-batch 
 of samples $\{\hat z_t^i\}_{i=1}^{N_t}$ of $z_t$, the approximate gradient is 
computed as $g_{t}(x_t) = (1/N_{t})\sum_{i=1}^{N_{t}}\nabla\ell_{t}(x_t,\hat z_{i})$; when $N_{t} = 1$ we have a ``greedy'' estimate and when $N_{t} > 1$ we have a ``lazy'' estimate~\cite{mendler2020stochastic}.
Accordingly, we consider an OSPGD described by:
\begin{align}
\label{eq:OSPGD}
x_{t+1} =\hat{G}_{t}(x_{t}), \,\,\,\,
\hat{G}_{t}(x) := \proj_{C_{t}}\left(x-\eta_{t}g_{t}(x)\right),
\end{align}
where $\eta_{t}>0$ is a stepsize. In the remainder, we focus on finding error bounds in the spirit of Theorem~\ref{thrm:OPGD} for~OSPGD.

\subsection{Bounds in expectation and high-probability}

Throughout our analysis, we interpret  OSPGD as an inexact OPGD with gradient error given by the random variable: 
\begin{equation}
\xi_{t} := \|g_{t}(x_{t})- \nabla f_{t}(x_{t},D_{t}(x_{t}))\|.
\end{equation}
We make the following assumption.

\vspace{.1cm}

\begin{assumption}{(\textbf{\textit{Sub-Weibull  Error}})}
\label{as:subWeibull}
The gradient error $\xi_{t}$ is sub-Weibull; i.e.,  $\xi_{t}\sim\sW(\theta,\nu_{t})$ for some $\theta, \nu_t > 0$. 
\end{assumption}
\vspace{.1cm}

Assumption~\ref{as:subWeibull} allows us to describe a variety of 
sub-cases, including scenarios where the error follows sub-Gaussian and
sub-Exponential distributions~\cite{Vershynin}, or any distribution 
with finite support. Further, notice that 
Assumption~\ref{as:subWeibull} does not require the random variables 
$\{\xi_t\}_{t \in \mathbb{N}_0}$ to be independent.
Examples of random variables that satisfy 
Assumption~\ref{as:subWeibull} are described in  
Section~\ref{sec:errormodel}.
Our error bounds for OSPGD are presented next.

\vspace{.1cm}

\begin{theorem}{(\textbf{\textit{Expected and High-probability Bounds for OSPGD}})}
\label{thrm:OSPGD}
Let Assumptions \ref{as:strongConvexity}-\ref{as:convexConstraints} hold, and suppose  that $\frac{\varepsilon_{t}\beta_{t}}{\alpha_{t}}<1$ for all 
$t \in \mathbb{N}_0$. Recall that 
$e_{t} = \|x_{t}-\bar{x}_{t}\|$.
Then, the following estimates hold for \eqref{eq:OSPGD}:
\begin{enumerate}
    \item For all $t \in \mathbb{N}$, 
    \begin{equation}
    \label{expecationBound}
    \EX\left[ e_{t+1} \right]\leq  a_{t}e_{0}+ \sum_{i=1}^{t}b_{i}(\varphi_i + \eta_i \EX [\xi_{i}] ) \, .
    \end{equation}
     \item If, additionally, Assumption~\ref{as:subWeibull} holds and $\delta\in(0,1)$, then with probability $1-\delta$:  
     \begin{equation}
    \label{highProbabilityBound}
    e_{t+1} \leq  \left(\frac{2e}{\theta}\right)^\theta \log^{\theta}\left(\frac{2}{\delta}\right) \big(a_{t}e_{0} + \sum_{i=1}^{t}b_{i}(\varphi_i + \eta_i \nu_i )\big),
     \end{equation}
\end{enumerate}
where $\{a_t\}$ and $\{b_i\}$ are as in Theorem \ref{thrm:OPGD}. 
\end{theorem}

\begin{proof}
Note that $x_t \in C_t$ for all $t \in \mathbb{N}$ directly follows by definition of Euclidean projection.
To show \eqref{expecationBound}, we first find a stochastic recursion. By the triangle inequality:
\begin{align*}
   e_{t+1} & \leq \| \hat{G}_{t}(x_t) - G_t(\bar{x}_t,D_{t}(\bar{x}_{t})) \| +\varphi_t \\     
    & \leq \| \hat{G}_{t}(x_t) - G_{t}(x_t,D_{t}(x_{t}))\| \\
    & + \|G_{t}(x_t,D_{t}(x_{t})) - G_{t}(\bar{x}_t,D_{t}(\bar{x}_{t})) \| + \varphi_t,
\end{align*}
where the second inequality follows by adding and subtracting
$G_{t}(x_t,D_{t}(x_{t}))$.
By iterating \eqref{eq:aux_Gbound}, we have
$\|G_{t}(x_t,D_{t}(x_{t})) - G_{t}(\bar{x}_t,D_{t}(\bar{x}_{t})) \| \leq \lambda_{t}e_{t}+\varphi_t$, where $\lambda_{t} := \rho_{t} + \eta_{t}\beta_{t}\varepsilon_{t}$,  
and thus 
     $e_{t+1}  \leq \eta_t\|g_t(x_t) - \nabla f_t(x_t,D_{t}(x_{t}))\| + \lambda_{t}e_{t}+\varphi_t$. 
This yields the stochastic recursion
    $e_{t+1}\leq \lambda_{t}e_{t}+\varphi_t + \eta_{t}\xi_{t}$. 
Expanding the recursion yields
\begin{align*}
e_{t+1} & \leq \left(\prod_{i=0}^{t} \lambda_{i}  \right)e_{0} + \varphi_{t}  + \sum_{i=0}^{t-1} \left ( \prod_{k=i+1}^{t} \lambda_{k} \right) (\varphi_{i} + \eta_i \xi_i),
\end{align*}
or, equivalently,
\begin{equation}
\label{randomUpperBound}
e_{t+1} \leq a_{t} e_{0} + \sum_{i=0}^{t}b_{i}(\varphi_i + \eta_i \xi_i ).
\end{equation}
Thus, \eqref{expecationBound} follows by taking the 
expectation on both sides.

To prove \eqref{highProbabilityBound}, we demonstrate that the righthand side of \eqref{randomUpperBound} is sub-Weibull distributed.
Since $\xi_{i}\sim\sW(\theta,\nu_{i})$, Proposition~\ref{prop:closureSubWeibull} implies that $b_{i}(\varphi_{i}+\eta_{i} \xi_{i})\sim \sW \left(\theta,b_{i}(\varphi_{i}+\eta_{i}\nu_{i})
\right)$. By summing over $i$, we obtain:
$$
\sum_{i=0}^{t} b_{i}(\varphi_{i}+\eta_{i} \xi_{i}) \sim \sW \left(\theta,\ \sum_{i=0}^{t}b_{i}(\varphi_{i}+\eta_{i}\nu_{i})
\right).
$$

Thus, by letting $\omega_t:=a_{t}e_{0}+ \sum_{i=0}^{t} b_{i}(\varphi_{i}+\eta_{i} \xi_{i})$, we conclude that 
$\omega_t 
\sim \sW \left(\theta,\upsilon_{t}\right)$, where
$\upsilon_{t} = \ a_{t}e_{0} + \sum_{i=0}^{t}b_{i}(\varphi_{i}+\eta_{i}\nu_{i})$.
From Definition \ref{def:subweibull} we have
\begin{equation}
    \mathbb{P}(|\omega_{t}| \geq \epsilon ) \leq 2\exp\left(-\frac{\theta}{2e}\left(\frac{\epsilon}{\upsilon_{t}}\right)^\frac{1}{\theta}\right). 
\end{equation}
Now let $\delta\in(0,1)$ be fixed and set it equal to the right hand side of the above inequality. Solving for $\epsilon$ yields $\epsilon= \log^{\theta}\left(\frac{2}{\delta}\right)\left(\frac{2e}{\theta}\right)^{\theta}\upsilon_{t}$. Then, we have that 
$\omega_{t} \leq \left(\frac{2e}{\theta}\right)^{\theta} \log^{\theta}\left(\frac{2}{\delta}\right)\upsilon_{t}$,
with probability $1\negat\delta$. 
Finally, \eqref{highProbabilityBound} follows by substitution.
\end{proof}

The bound \eqref{expecationBound} generalizes the 
estimate in Theorem \ref{thrm:OPGD} by accounting for the gradient error. It is also worth pointing out that \eqref{expecationBound} and  \eqref{highProbabilityBound} have a similar structure;  indeed, \eqref{highProbabilityBound} differs only by a logarithmic factor and by the introduction of the tail parameters $\nu_{i}$ (which replaces the  expectation term).

\begin{remark}
An alternative high probability bound 
can be obtained by using \eqref{expecationBound} and Markov's inequality. For any 
$\delta\in(0,1)$, then Markov's inequality guarantees that:
\begin{equation}
\label{markovBound}
 e_{t+1}\leq \frac{1}{\delta}\bigg( a_{t}e_{0} + \sum_{i=1}^{t}b_{i}(\varphi_{i} + \eta_i \EX [e_i] ) \bigg), 
\end{equation}
with  probability at least $1-\delta$. However, if we increase the confidence of the bound by allowing $\delta\rightarrow 0$, the right-hand-side of~\eqref{markovBound} grows more rapidly than~\eqref{highProbabilityBound}.
\QEDB
\end{remark}

Note that the bounds in Theorem~\ref{thrm:OSPGD} are valid for any $t \in \mathbb{N}$. The asymptotic behavior is noted in the next remark. 

\begin{remark}
If~\eqref{stepSize} holds, then $\limsup_{t \rightarrow +\infty} e_{t} \leq (1 - \tilde{\lambda})^{-1}( \tilde{\varphi} + \tilde{\eta} \tilde{\xi} )$ \emph{almost surely}, where $\tilde{\eta}$ and $\tilde{\xi}$ are upper bounds on the step size and $\EX [\xi_{t}]$; the proof is omitted because of space limits, but follows arguments similar to~\cite{bastianello2021stochastic}.  
\QEDB
\end{remark}

\subsection{Remarks on the error model}
\label{sec:errormodel}

The class of sub-Weibull distributions allows one to consider variety of error models. For instance, it includes sub-Gaussian and sub-exponential as sub-cases by setting $\theta = 1/2$ and $\theta = 1$, respectively. We notice that a sub-Gaussian assumption was typically utilized in prior works on stochastic gradient descent; for example, the assumption $\mathbb{E}[\exp\left(\xi^2/\sigma^2\right)]\leq e$ in~\cite{nemirovski2008robust} corresponds to sub-Gaussian tail behavior. However, recent works suggest that stochastic gradient descent may exhibit errors with tails that are heavier than a sub-Gaussian (see, e.g.,~\cite{hodgkinson2020multiplicative}). To further elaborate on the flexibility offered by a sub-Weibull model, we provide the following additional examples. 

\begin{example} 
Suppose that each entry of the gradient error $g_{t}(x_{t})- \nabla f_{t}(x_{t},x_{t})$ follows a distribution $\sW(\theta, \nu)$, $i = 1, \ldots, d$ for given $\theta, \nu > 0$. Then $\|\xi_t\|$ is sub-Weibull with $\|\xi_t\| \sim \sW(\theta, 2^\theta \sqrt{d} \nu)$ \cite{bastianello2021stochastic}.
\QEDB
\end{example}

\begin{example} 
Suppose that an entry of the gradient error $g_{t}(x_{t})- \nabla f_{t}(x_{t},x_{t})$ is Gaussian is zero mean and variance $\varsigma^2$; then, it it sub-Gaussian with  sub-Gaussian norm $C \varsigma$, with $C$ an absolute constant~\cite{Vershynin}, and it is therefore a sub-Weibull $\sW(1/2, C^\prime \varsigma)$ with $C^\prime$ an absolute constant.
\QEDB
\end{example}

\begin{example} 
Suppose that $\xi_t$ is a random variable with mean $\mu_{t} := \EX[\xi_{t}]$, such that $\xi_t \in [\bar{\xi}, \underline{\xi}]$ almost surely. Then $\xi_t - \mu_{t} \sim \sW(1/2, (\underline{\xi} - \bar{\xi}) / \sqrt{2})$~\cite{bastianello2021stochastic}.
\QEDB
\end{example}

\section{Application to Electric Vehicle Charging}
\label{sec:results}
This section illustrates the use of the proposed 
algorithms in an application inspired from  
\cite{tushar2012economics}, where the operator of a fleet of electric vehicles (EVs) seeks to determine an optimal charging policy in order to minimize its charging costs.  The region of interest is modeled as a graph 
$\mc G = (\mc V, \mc E)$, where each node in $\mc V$ 
represents a charging station (or a group thereof), and an 
edge $(i,j)$ in $\mc E$ allows vehicles to transfer 
from node $i$ to $j$. We assume that the graph is strongly connected, so that  EVs can be redirected from one node to any other node. We let $x_{i} \in \realpos$ denote the energy 
requested by the fleet at node $i \in \mc V$.
We assume that the net energy available is limited, and define the set 
$C_t := \{x \in \mathbb{R}^d: \sum_{i \in \mc V} x_{i} \leq c_t\}$, for a given  
$c_t \in \real_{>0}$.
Given $\{x_{i}\}$, the operator of the power grid 
strategically chooses a price per unit of 
energy so as to optimize its revenue from selling energy; 
we let $z_{i} \in \realpos$ denote the selected price in 
region $i$, and we hypothesise that 
$z_i\sim\mathcal{N}(\mu_{t} x_i,\sigma_{t}^{2})$, 
$\mu_{i,t}, \sigma_t \in \realpos$ as an example. We note that, although the grid operator can choose the 
price arbitrarily large to maximize its revenue, large 
prices may compel the fleet operator to withdraw its demand, thus motivating the use of a model where the mean grows linearly with the energy demand. Accordingly, we model the cost function of 
the EV operator as follows~\cite{tushar2012economics}:
\begin{align}
\label{eq:costFunctionApplication}
\ell_t(x,z) = \sum_{i \in \mc V} z_{i} x_{i,t} 
- \gamma_{i,t} x_{i} 
+ \kappa_{i,t} x_{i}^2,
\end{align}
where $\gamma_{i,t} \in \realpos$, models the charging 
aggressiveness of the fleet operator, and  
$\kappa_{i,t} x_{i,t}^2$ quantifies the satisfaction the fleet
operator achieves from consuming one unit of energy.
In \eqref{eq:costFunctionApplication}, the term 
$z_{i,t} x_{i,t}$ describes the charging cost at station 
$i$, the quantity $\gamma_{i,t} x_{i,t}$, and models the 
energy demand at the $i$-th station.
Notice that, because the displacement of vehicles can 
change over time, we assume that the parameters 
$\gamma_{i,t}$ and $\xi_{i,t}$ are time dependent.
We note that: (i) because of the capacity constraint 
$x_t \in C_t$, the decision variables  
$x_{i,t}, i \in \mc V,$  are coupled, and (ii) although the 
optimization could be solved in a distributed fashion since 
\eqref{eq:costFunctionApplication} is separable, our focus 
is to solve it in a centralized way since the EV operator is
unique.

We apply the proposed methods to a system of 10 homogeneous charging stations over 100 time steps with fixed net energy ($c_{t}=10$). Namely, $\gamma_{i,t}=-1/100| t-50| + 1$ and $\kappa_{i,t} = 2$ for $i\in\{1,\hdots,10\}$. The charging cost distribution is informed by $\mu_{t}$ and $\sigma_{t}$; in our case, $\mu_{t}$ is the time series data of CAISO real-time prices deposited in Fig \ref{fig:muValues}  (taken from \url{http://www.energyonline.com}) and $\sigma_{t}=1$. Given these parameter values, the cost is $\alpha_{t}$-strongly convex and $\beta_{t}$-jointly smooth with $\alpha_{t}=\beta_{t}=2$. Following the results in \cite{givens1984metrics}, the distributional maps are $\varepsilon_{t}$-sensitive with $\varepsilon_{t}=\mu_{t}$. The sequence of performatively stable points are computed in closed form by solving the KKT equations. 

\begin{figure}[tb]
    \centering
    \includegraphics[width=\columnwidth]{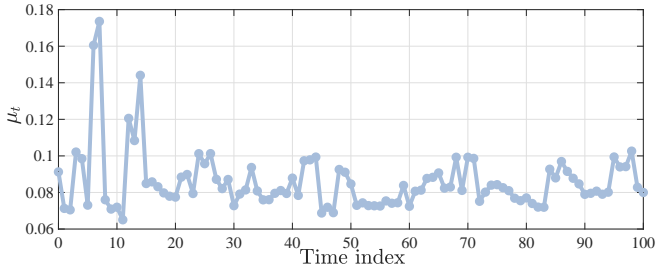}
    \caption{Time series data representing the price of energy in dollars per kilowatt hour (kWh). Each time step represents 5 minutes.}
    \label{fig:muValues}
\end{figure}
\begin{figure}
    \centering
    \includegraphics[width=\columnwidth]{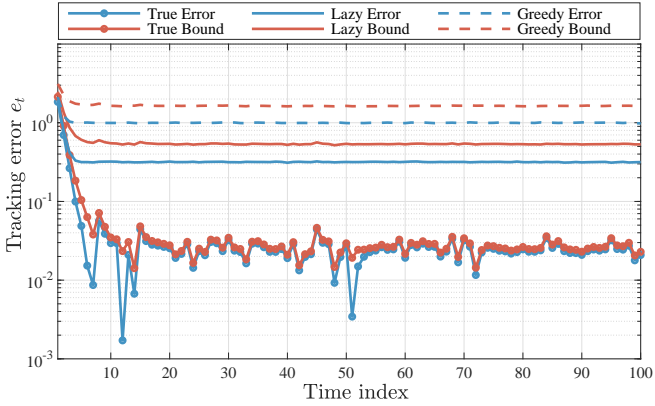}
    \caption{Performance comparison of OPGD and OSPGD.}
    \label{fig:performance}
\end{figure}

For each experiment, we run OPGD and OSPGD with fixed step size $\eta_{t} =0.3$ by drawing initial state $x_{0}$ uniformly from a sphere of radius $5$. For OSPGD, we compute the mean tracking error for both greedy and lazy deployments. The mean tracking error for each is computed via Monte Carlo simulation using $1,000$ realizations of the initial state. 

In Fig.~\ref{fig:performance}, we illustrate the tracking 
errors and corresponding upper bounds presented in 
Theorems~\ref{thrm:OPGD} and~\ref{thrm:OSPGD}. ``True'' (i.e., true gradient) refers to the OPGD, ``greedy'' to the OSPGD with $N_t = 1$, and ``lazy'' to the OSPGD with $N_t = 10$. We notice that the upper bound curve mimics the 
evolution of the tracking error; yet, in the instance of 
OSPGD the relationship is looser relative to the OPGD 
curves.

\section{Conclusions}
\label{sec:conclusions}
This paper considered online gradient and stochastic gradient methods for tracking solutions of time-varying stochastic optimization problems with decision-dependent  distributions. Under a distributional sensitivity assumption, we derived 
explicit error bounds for the two methods.
In particular, we derived convergence   in expectation and in high probability for the OSPGD by assuming that the error in the gradient follows a sub-Weibull distribution. To the best of our knowledge, our convergence results for online gradient methods are the first in the literature for time-varying stochastic optimization problems with decision-dependent  distributions.

\bibliographystyle{IEEEtran}
\bibliography{references.bib}

\end{document}